\documentclass[12pt, leqno]{amsart}
\usepackage[
	style=numeric,
	citestyle=numeric,
	backend=biber,
    maxbibnames=5,
    doi=false,isbn=false,url=false
]{biblatex}
\renewbibmacro{in:}{%
  \ifentrytype{article}{}{\printtext{\bibstring{in}\intitlepunct}}}
\usepackage[T1]{fontenc}
\usepackage{a4, amsmath}
\usepackage{mathtools}
\usepackage{amssymb}
\usepackage{amsthm, amscd, mathdots}
\usepackage{enumerate}
\usepackage{hyperref}
\usepackage{cleveref}
\usepackage{datetime}
\usepackage{xcolor}

\usepackage[top=1in, bottom=1in, left=1in, right=1in]{geometry}

\theoremstyle{definition}
\newtheorem{defi}{Definition}[section]
\newtheorem{propdef}[defi]{Proposition and Definition}
\theoremstyle{plain}
\newtheorem{prop}[defi]{Proposition}
\newtheorem{lem}[defi]{Lemma}
\newtheorem{thm}[defi]{Theorem}
\newtheorem*{thm*}{Theorem}
\newtheorem*{prop*}{Proposition}
\newtheorem{cor}[defi]{Corollary}

\newtheorem{ques}[defi]{Question}
\theoremstyle{remark}
\newtheorem{rem}[defi]{Remark}
\newtheorem{eg}[defi]{Example}

\crefname{thm}{theorem}{theorems}
\crefname{lem}{lemma}{lemmata}
\crefname{prp}{proposition}{propositions}
\crefname{cor}{corollary}{corollaries}

\newcommand{\mc}{\mathcal}

\newcommand{\mbb}{\mathbb}
\newcommand{\nat}{\mbb N}
\newcommand{\cc}{\mathbb C}
\newcommand{\rr}{\mathbb R}
\newcommand{\qq}{\mbb Q}

\newcommand{\ff}{\mbb F}
\newcommand{\zz}{\mbb Z}
\newcommand{\co}{\mc O}

\makeatletter
\newcommand*{\house}[1]{%
	\mathord{%
		\mathpalette\@house{#1}%
	}%
}
\newcommand*{\@house}[2]{%
	\dimen@=\fontdimen8 %
	\ifx#1\scriptscriptstyle\scriptscriptfont
	\else\ifx#1\scriptstyle\scriptfont
	\else\textfont\fi\fi
	3 %
	\sbox0{%
		$#1%
		\vrule width\dimen@\relax
		\overline{%
			\kern2\dimen@
			\begingroup % to keep changes of \dimen@ in #2 local
			#2%
			\endgroup
			\kern2\dimen@
		}%
		\vrule width\dimen@\relax
		\mathsurround=1.5\dimen@ % outside margin
		$%
	}%
	\ht0=\dimexpr\ht0-\dimen@\relax
	\dp0=\dimexpr\dp0+2\dimen@\relax
	\vbox{%
		\kern\dimen@ % reinsert previously removed space
		\copy0 %
	}%
}

\DeclareMathOperator{\Frac}{Frac}

\DeclareMathOperator{\Tr}{Tr}
\DeclareMathOperator{\Norm}{N}
\DeclareMathOperator{\Gal}{Gal}

\DeclareMathOperator{\sgn}{sgn}

\newcommand{\tpu}{\mc{O}^{+, \times}}
\newcommand{\Rc}{R_{\operatorname{class}}}

\newcommand{\NPI}{NPI}
\newcommand{\multi}[2][]{
\ifx &#1&
{\qq^{[#2]}}
\else
{{#1}^{[#2]}}
\fi
}
\newcommand{\multiBZ}[2][]{
\ifx &#1&
{\qq^{(#2)}}
\else
{{#1}^{(#2)}}
\fi
}
\newcommand{\multiBZab}[2][]{
\ifx &#1&
{\qq_{ab}^{(#2)}}
\else
{{#1}_{ab}^{(#2)}}
\fi
}

\title{Universal quadratic forms and Northcott property of infinite number fields}
\author{Nicolas Daans}
\author{Vítězslav Kala}
\author{Siu Hang Man}
\address{Charles University, Faculty of Mathematics and Physics, Department of Algebra, Sokolov\-sk\' a 83, 186~75 Praha~8, Czech Republic}
\email[N.~Daans]{nicolas.daans@matfyz.cuni.cz}
\email[V.~Kala]{vitezslav.kala@matfyz.cuni.cz}
\email[S.~H.~Man]{shman@karlin.mff.cuni.cz}
\thanks{The authors were supported by Czech Science Foundation GA\v CR, grant 21-00420M. N.D.~and S.H.M.~were supported by Charles University programme PRIMUS/24/SCI/010. S.H.M.~was also supported by OP RDE project
	No. CZ.02.2.69/0.0/0.0/18\_053/0016976 International mobility of research, technical and administrative staff at the
	Charles University. \\
This is an author accepted manuscript (AAM). This manuscript was published in the Journal of the London Mathematical Society and the published version may be found here: \url{https://doi.org/10.1112/jlms.70022}.
	}

\begin{document}
\begin{abstract}
We show that if a universal quadratic form exists over an infinite degree, totally real extension of the field of rationals $\qq$, then the set of totally positive integers in the extension does not have the Northcott property. 
In particular, this implies that no universal form exists over the compositum of all totally real Galois fields of a fixed prime degree over $\qq$. Further, by considering the existence of infinitely many square classes of totally positive units, we show that no classical universal form exists over the compositum of all such fields of degree $3d$ (for each fixed odd integer~$d$).
\end{abstract}

\makeatletter
\@namedef{subjclassname@2020}{%
  \textup{2020} Mathematics Subject Classification}
\makeatother

\maketitle

\section{Introduction}

The study of representations of integers as values of quadratic forms has history going back to the ancient times. In the last century it led to classification works of universal forms such as Ramanujan's, and to the $290$-theorem of Bhargava--Hanke \cite{Bhargava-Hanke}.

A natural generalization asks, for a given totally real number field $K$ with the ring of integers $\mc O_K$, whether a positive definite quadratic form defined over $\mc O_K$ is \textit{universal}, that is, represents all totally positive elements in $\mc O_K$. Besides Lagrange's theorem on the universality of the sum of four squares $x^2+y^2+z^2+w^2$ over $\qq$, Maa{\ss}  \cite{Maass} established the universality of the sum of three squares $x^2+y^2+z^2$ over $\qq(\sqrt{5})$. 
Although Siegel \cite{Siegel} proved that these are the only two fields in which the sum of any number of squares is universal,
some universal quadratic form exists for every totally real number field $K$: this readily follows from the asymptotic local-global principle established in \cite{HKK}, see e.g.~\cite[Section 5]{KalaSurvey} for a sketch of a proof.
It is hence natural to ask what the minimal number of variables of such a universal quadratic form is. It is known that this minimal number of variables can be arbitrarily large; in fact, for any natural number $r$, there are infinitely many quadratic  \cite{Kala-no-universal} and cubic \cite{Ya} number fields over which there exists no universal quadratic form in $r$ variables.
Recent work of Kala--Yatsyna--Żmija refines this result to show that the real quadratic fields admitting a universal quadratic form with a given number of variables have density zero \cite{Kala-Yatsyna-Zmija}, which was then further extended to multiquadratic fields by Man \cite{Man}.

In this paper, we consider a further generalization of the problem, and study quadratic forms over totally real algebraic field extensions of $\qq$, possibly of infinite degree (this is what we mean by `infinite number fields' in the title). Here, 
it is very unclear whether any universal forms exist at all.

We pioneer two ways of approaching the problem that are based on Northcott numbers, and on square classes of totally positive units.
Our main results then show that indeed, universal forms need not exist and that this, in fact, happens quite frequently.

Let us introduce a few notations and concepts.
We shall work with \emph{quadratic lattices}, which provide a slight generalization of quadratic forms, see \Cref{sect:universal-lattices} for a definition. 
Further, for $d \in \nat$ (here $\nat$ denotes the set of positive integers), we denote by $\multiBZ{d}$ the compositum of all number fields of degree at most $d$, and by $\multi{d}$ the compositum of all totally real Galois extensions of $\qq$ of degree exactly $d$.
For example, one has $\multi{2} = \qq(\sqrt{n} \mid n \in \nat)$ and $\multiBZ{2} = \qq(\sqrt{n} \mid n \in \zz) = \multi{2}[\sqrt{-1}]$.

\begin{thm}\label{thm:newmain}
There are no universal quadratic lattices over any infinite extension of $\qq$ contained in $\multi{q}$, where $q=p$ or $q = p^2$ for a prime number $p$.
\end{thm}
In fact, we prove non-existence of universal quadratic lattices for a larger class of totally real fields, see \Cref{thm:newmain-general} for the precise statement.

The proof is based on considering the Northcott property for totally positive integers.
First recall that the \textit{house} of $\alpha\in K$ is defined as
$\house{\alpha}= \max_i(|\alpha_i|)$, where 
 $\alpha_1,\dots,\alpha_n$ are all the conjugates of  $\alpha$.
We will say that the set of totally positive integers $\mc{O}_K^+$ has the \emph{Northcott property} (with respect to the house) if, for every $r \in \rr$, there exist only finitely many $\alpha \in \mc{O}_K^+$ with $\house{\alpha} < r$. 
For short, in this Introduction we will also sometimes say that $K$ has the \emph{Northcott property for Positive Integers} (\NPI).

The Northcott property, and the related Northcott numbers and the Julia Robinson (JR) property (see \Cref{sect:northcott} for definitions), have been used extensively in the past as a way of classifying infinite extensions of $\qq$, and of extending properties from number fields to them, e.g.~in the context of decidability.
For additional context, recall that
Robinson \cite{Rob62} 
showed that, whenever $K$ is an algebraic extension of $\qq$ such that $\mc{O}_K^+$ has the JR property, then the semiring $(\nat, +, \cdot)$ is first-order definable in $\mc{O}_K$ (and thus $\co_K$ has undecidable first-order theory). Perhaps coincidentally, her proof used the fact that every totally positive element of $K$ is the sum of four squares of elements of $K$. 
Specifically, she showed that this holds for $\multi{2}$ (which has \NPI) 
or for the field of all totally real numbers (which does not have \NPI).
In recent years, her works have inspired new developments in decidability in number theory; see e.g.~\cite{Vidaux-Videla15} for a study of rings with and without the JR property,  \cite{Vidaux-Videla} for the undecidability of the first-order theory of the maximal abelian subfield of $\multi{d}$ for every $d \in \nat$ (by using \NPI), or \cite{U2,U3,U1} for the undecidability of the first-order theory of certain non-real infinite algebraic extensions of $\qq$, using Northcott property and Northcott numbers.
We mention also the recent preprint \cite{MRS}, which contains both a new uniform approach to the first-order definability of $\nat$ in infinite extensions, and discusses the history of the study in more detail.

It is known that the fields $\multi{p}$ and $\multi{p^2}$ have \NPI\ for any prime $p$; and more generally, it follows from the work of Bombieri--Zannier  \cite{Bombieri-Zannier} that, for any $d \in \nat$, any abelian subfield of $\multiBZ{d}$ has \NPI\ 
(although we work with the house for concreteness, this follows from their result on the Weil height; see also \cite{Widmer} for a recent survey).
Furthermore, clearly any subfield of a field with \NPI\ also has \NPI.
Hence we obtain \Cref{thm:newmain} as a consequence of the following.
\begin{thm}\label{thm:univ-Northcott}
Let $K$ be a totally real infinite  extension of $\qq$.
If $\mc{O}_K^+$ has the {Northcott property} with respect to the house, then there exists no universal quadratic lattice over $K$.
\end{thm}

It is conjectured that $\multiBZ{d}$ has the Northcott property with respect to the Weil height for every $d \in \nat$, from which it would follow that $\multiBZ{d}$ (and thus all of its subfields) have \NPI\ for every $d \in \nat$, but this conjecture remains a hard open question.

\Cref{thm:univ-Northcott} will be proved as \Cref{thm:3.1}, where we actually prove a slightly more general statement that applies also to subrings $\mathcal O\subseteq \mathcal O_K$.
To briefly sketch the idea, suppose that $(\Lambda,Q)$ is a universal lattice over $K$ and let us naturally extend the house to $\Lambda$  (concretely, for $v=(v_1,\dots,v_n)\in\Lambda\subseteq K^n$ we set $\house{v}=\max(\house{v_1},\dots,\house{v_n})$). We then have that if $Q(v)=\alpha$, then $\house{v}$ is bounded above in terms of $\house{\alpha}$. Thus in order for $Q$ to  represent an element $\alpha$, the representing vector $v$ needs to have small heights of its coordinates, from which we can establish the existence of infinitely many totally positive elements of $\mc O_K$ with bounded height.

Let us also note that, somewhat curiously, every number field has \NPI, 
despite the fact that every number field admits a universal quadratic lattice. However, from the proof of \Cref{thm:univ-Northcott} it is clear that one indeed crucially needs to use the infinite degree assumption.

Very recently, Prakash generalized the arguments given in this paper to show that an analogue of \Cref{thm:univ-Northcott} holds for forms and lattices of higher degree \cite{Prakash}.

\medskip

In an alternative approach, we study the behavior of the units, towards an application to determining the existence of universal \emph{classical} lattices, i.e.~such that all the values of the associated symmetric bilinear form lie in $\mc O_K$. Let us denote by $\mc{O}_K^\times$ the group of invertible integers of $K$ (i.e.~units), by $\mc{O}_K^{\times 2}$ the subgroup of square units, and $\tpu_{K} = \mc{O}_K^+ \cap \mc{O}_K^\times$. Precisely, we study the quotient $\tpu_K/\mc{O}_K^{\times 2}$. When $K$ is a number field, the quotient $\tpu_K/\mc{O}_K^{\times 2}$ is always finite and its size is closely related to the narrow class group of $K$, whose behavior is a focus of much research, see for example \cite{SignatureRanks1,SignatureRanks2,SignatureRanks3,DV18,Stevenhagen}. 

When $K$ is an algebraic field of infinite degree, then the quotient $\tpu_K/\mc{O}_K^{\times 2}$ can be infinite. And when $K$ is totally real, the infinitude of the quotient $\tpu_K/\mc{O}_K^{\times 2}$ has an interesting consequence concerning universal lattices.

\begin{prop}\label{prop:intro}
Let $K$ be a totally real extension of $\qq$. If $\lvert \tpu_K/\mc{O}_K^{\times 2} \rvert$ is infinite, then there exists no classical universal quadratic lattice over $K$.
\end{prop}
The argument for this is quite standard: it is well-known that if a classical quadratic lattice $(\Lambda, Q)$ over a number field $K$ represents a unit in $\mc O_K$, say $Q(v) \in \mc O_K^\times$, then we may split the quadratic lattice as an orthogonal sum $\Lambda = \mc O_K v \perp \Lambda'$ for a sublattice $\Lambda' \subseteq \Lambda$. This also works for fields of infinite degree, with essentially the same proof (see Propositions \ref{prop:splitting-units} and \ref{prop:units-no-classical}).

This observation yields a different route to establishing examples of infinite degree fields without classical universal quadratic lattices.
While most of the examples we can construct are covered by \Cref{thm:newmain}, we also obtain new ones. Specifically, in Theorems \ref{thm:Q2} and \ref{thm:Q3} we show the following result.

\begin{thm}\label{thm:main}
Let $K = \multi{2}$, or $K = \multi{3d}$ for some odd positive integer $d$.
Then $\tpu_K/\mc{O}_K^{\times 2}$ is infinite.
In particular,
there are no universal classical quadratic lattices over $\multi{2}$ or $\multi{3d}$ for any odd positive integer $d$.
\end{thm}
In the recent preprint \cite[Theorem 2]{DKMWY} we show, together with Widmer and Yatsyna, that the restriction to `classical' universal quadratic lattices in \Cref{prop:intro}, and consequently in \Cref{thm:main}, may in fact be dropped. The argument for non-classical lattices is noticeably more involved.

We observe that there also exist totally real fields of infinite degree with universal quadratic forms. 
Perhaps the easiest example is given by $K = \qq^{\operatorname{tr}}$, defined as the union of all totally real finite extensions of $\qq$. As the square root of any totally positive element of $\mc O_K$ is totally real, it again lies in $\mc O_K$, because $\mc O_K$ is integrally closed, and so $X^2$ is a universal form over $K$.
Another example is given by the \emph{Pythagorean closure} of $\qq$ (where again $X^2$ is universal), see \Cref{eg:pythagorean-closure} for details. Further, in \Cref{prop:universal-qf-of-degree-2power} we construct, for any natural number $n$, fields over which the smallest rank of a universal quadratic lattice is $2^n$. 

Our results leave many questions open for further research.
Specifically, it is tantalizing to consider whether converses to \ref{thm:univ-Northcott} and \ref{prop:intro} may hold. In both cases, we suspect that the answer will be `No', although we were not able to construct a counterexample for either of these.
We discuss two other open questions in Subsection \ref{sec:ques}.

\subsection*{Acknowledgements}
We would like to thank Arno Fehm, Fabien Pazuki, Om Prakash, Xavier Vidaux, Carlos Videla, and Martin Widmer for helpful comments on earlier drafts of the preprint.
We thank the anonymous referee for their careful proofreading, which led to the correction of several typos and small mistakes, and for pointing out to us certain references.

\section{Quadratic lattices over totally real fields}\label{sect:universal-lattices}

As we will be working with infinite extensions of $\qq$, let us carefully define the notions we will use.
Let $\alpha\in\cc$ be algebraic over $\qq$. A \textit{conjugate} of $\alpha$ is a root of the minimal polynomial for $\alpha$. 
The element $\alpha$ is \textit{totally real} if all its conjugates are real, and \textit{totally positive} if all its conjugates are positive real numbers.
An algebraic extension $K/\qq$ is \textit{totally real} if all elements of $K$ are totally real.
Throughout the paper, by a \textit{totally real field} $K$ we will mean a totally real algebraic extension of $\qq$, of finite or infinite degree.

Note that, clearly, an algebraic extension $K$ is totally real if and only if each subfield $K\supseteq F\supseteq \qq$ of finite degree $[F:\qq]$ is totally real, which can be characterized alternatively by the fact that all the complex embeddings $F\hookrightarrow\cc$ are real.

For an algebraic field extension $K/\qq$, 
$\mc{O}_K$ denotes its ring of integers (i.e.~of elements that are integral over $\zz$).
For a subset $S\subseteq K$, $S^+$ denotes all the totally positive elements of $S$, and $S^{+, \times}$ denotes the set of non-zero totally positive elements of $S$ whose multiplicative inverse also lies in $S$.

To phrase our results in a coordinate-free way, we will make use of so-called (quadratic) lattices, which generalize quadratic forms. We refer to \cite{OMe00} for an extensive exposition. We summarize some aspects here, partially for the convenience of the reader, but also because several sources work under the assumption that $K/\qq$ is a finite extension.

Let $\mc O \subseteq \mc O_K$ be a subring such that $\Frac \mc O = K$. An \emph{$\mc O$-lattice} is a finitely generated $\mc O$-module $\Lambda \subseteq V$, where $V$ is a finite dimensional $K$-vector space. The \emph{rank} of $\Lambda$ is the $K$-dimension of the $K$-span of $\Lambda$; it is a non-negative integer.
A \emph{quadratic map} on $V$ is a map $Q : V \to K$ such that $Q(av) = a^2Q(v)$ for all $a \in K$ and $v \in V$, and such that the map
$$ B_Q: V \times V \to K : (v, w) \mapsto \frac{1}{2}(Q(v+w) - Q(v) - Q(w))$$
is bilinear. If we endow the lattice $\Lambda$ with the restriction of a quadratic map $Q$ to $\Lambda$ (still denoted $Q$), the pair $(\Lambda, Q)$ is called a \emph{quadratic $\mc O$-lattice}.
We sometimes write simply \emph{quadratic lattice} for quadratic $\mc O_K$-lattice when the field $K$ is clear from context.

If $(\Lambda, Q)$ is a quadratic $\mc{O}_K$-lattice and $\Lambda$ is a \textit{free} $\mc{O}_K$-module, we may pick a basis $(x_1, \ldots, x_n) \subseteq \Lambda$ for $\Lambda$.
An arbitrary element of $\Lambda$ can then be written as $a_1x_1 + \ldots + a_nx_n$ for certain $a_i \in \mc{O}_K$, and we compute that
\begin{align*}
Q(a_1x_1 + \ldots + a_nx_n) &= \sum_{i=1}^n a_i^2Q(x_i) + 2\sum_{1 \leq i < j \leq n} a_ia_j B_Q(x_i, x_j) \\
&= F(a_1, \ldots, a_n)
\end{align*}

where $F( X_1, \ldots,  X_n)$ is the quadratic form 
\begin{align*}
F( X_1, \ldots,  X_n) = \sum_{i=1}^n  X_i^2Q(x_i) + 2\sum_{1 \leq i < j \leq n}  X_i X_j B_Q(x_i, x_j) \in K[ X_1, \ldots,  X_n].
\end{align*}
Conversely, to an arbitrary quadratic form $F( X_1, \ldots,  X_n)$ over $K$, we may associate the quadratic lattice $(\mc{O}_K^n, Q_H : \mc{O}_K^n \to K, \; (a_1, \ldots, a_n) \mapsto F(a_1, \ldots, a_n))$.
The theory of quadratic lattices as such generalizes the theory of quadratic forms over $\mc{O}_K$, and we may apply properties and definitions developed for quadratic lattices also to quadratic forms via this translation.

Consider again some arbitrary subring $\mc O \subseteq \mc{O}_K$ with $\Frac \mc O = K$, and let $(\Lambda, Q)$ be a quadratic $\mc O$-lattice.
An element $d \in K$ is \emph{represented} by $Q$ if $d = Q(v)$ for some $v \in \Lambda$. We say that the quadratic $\mc O$-lattice is \emph{integral} if $Q(\Lambda) \subseteq \mc O$.
We further call an integral quadratic $\mc O$-lattice $(\Lambda, Q)$
\begin{itemize}
\item \emph{positive definite} if $Q(v)$ is totally positive for every $0\ne v \in \Lambda$,
\item \emph{universal} if it is positive definite and represents every totally positive element of $\mc O$,
\item \emph{classical} if $B_Q(v, w) \in \mc O$ for all $v, w \in \Lambda$.
\end{itemize}
Note that for a quadratic form $F(X_1, \ldots, X_n) = \sum_{1 \leq i \leq j \leq n} b_{i j}X_iX_j$ with $b_{i j} \in K$, the associated quadratic lattice is integral if and only if $b_{i j} \in \mc{O}_K$ for all $1 \leq i \leq j \leq n$, and classical if and only if additionally $b_{i j} \in 2\mc{O}_K$ for $1 \leq i < j \leq n$.

When $(\Lambda, Q)$ is a quadratic $\mc O$-lattice and $\Lambda_1, \Lambda_2$ are $\mc O$-submodules of $\Lambda$, we write $\Lambda = \Lambda_1 \perp \Lambda_2$ if $\Lambda = \Lambda_1 \oplus \Lambda_2$ and $B_Q(v, w) = 0$ for all $v \in\Lambda_1$ and $w \in \Lambda_2$.
In this case, for any $v \in \Lambda_1$ and $w \in \Lambda_2$ we have $Q(v+w) = Q(v) + Q(w)$.

Finally, let $L/K$ be an extension of totally real fields.
If $(\Lambda, Q)$ is a quadratic $\mc O$-lattice for some subring $\mc O \subseteq \mc{O}_K$, and $\mc O' \subseteq \mc O_L$ is a subring such that $\Frac \mc O' = L$ and $\mc O \subseteq \mc O'$, then we can associate to it a quadratic $\mc O'$-lattice, which we denote by $(\Lambda_{\mc O'}, Q_L)$ and call the \emph{scalar extension of $(\Lambda, Q)$ to $\mc O'$}.
As an $\mc O'$-module, we have $\Lambda_{\mc O'} = \Lambda \otimes_{\mc O} \mc O'$, and $Q_L$ is the unique quadratic map on $\Lambda_{\mc O'}$ with the property that $Q_L(v \otimes a) = a^2Q(v)$ for all $v \in \Lambda$, $a \in L$.

Conversely, given a subring $\mc{O}'$ of $\mc{O}_L$ with $\Frac \mc{O}' = L$ and a quadratic $\mc O'$-lattice $(\Lambda, Q)$, we say that $\Lambda$ is \emph{defined over} $K$ if $\mc O'|_K = \mc O'\cap \mc O_K$ satisfies $\Frac (\mc O'|_K) = K$, and there is an $\mc O'|_K$-lattice $(\Lambda_{\mc O'|_K}, Q_K)$ such that $(\Lambda,Q)$ is its scalar extension to $\mc O'$. In this case, several properties of $\Lambda$ can be inferred from the properties of $\Lambda_{\mc O'|_K}$, which can be convenient.

\begin{prop}\label{prop:defined-over}
Let $L/\qq$ be an algebraic field extension, $\mc{O}$ a subring of $\mc{O}_L$ with $\Frac \mc{O} = L$, and $(\Lambda, Q)$ a quadratic $\mc O$-lattice. Then $(\Lambda, Q)$ is defined over some number field $K \subseteq L$.
\end{prop}

\begin{proof}
Let $(\Lambda, Q)$ be a quadratic $\mc O$-lattice of rank $r$, generated by $v_1,\ldots, v_k \in \Lambda$. We have $k \geq r$, and without loss of generality $\lbrace v_1, \ldots, v_r \rbrace$ forms a maximal linearly independent subset of $\Lambda$. We may consider the $L$-vector space $V$ containing $\Lambda$ and such that the vectors $v_1,\ldots, v_r$ form a basis of $V$.
We must then have linear relations
$$v_i = c_{i1} v_1 + \ldots + c_{ir} v_r \ \text{ for } \ r+1 \le i \le k, \ \ c_{ij} \in L.$$ 

Since $\Frac \mc O = L$, we may write $c_{ij} = s_{ij}/t_{ij}$ with $s_{ij},t_{ij}\in \mc O$. Write $B_Q(v_i, v_j) = b_{ij} \in L$. Then we may again write $b_{ij} = u_{ij}/w_{ij}$ with $u_{ij},w_{ij} \in \mc O$. Let $K$ be the subfield of $L$ generated by the elements $s_{ij},t_{ij},u_{ij},w_{ij}$. Since there are only finitely many of these, we infer that $K$ is a number field, and $\mc O|_K = \mc O \cap \mc O_K$ satisfies $\Frac \mc{O}|_K = K$. Let $\Lambda_{\mc O|_K}$ be the $\mc O|_K$-submodule of $\Lambda$ generated by $v_1, \ldots, v_k$, let $V_K$ be the $K$-subspace of $V$ generated by $v_1, \ldots, v_k$, and set $Q_K = Q|_K$.
We have $\dim_K(V_K) = r$ and $$\Lambda_{\mc O|_K} \otimes_{\mc O|_K} L \cong V_K \otimes_{K} L\cong V\cong\Lambda\otimes_{\mc O} L.$$
It restricts to an isomorphism of $\mc O$-modules $\Lambda_{\mc O|_K} \otimes_{\mc O|_K} \mc O \to \Lambda$ which additionally respects the map $Q$.
Thus $\Lambda$ is defined over the number field $K$.
\end{proof}

Let us give an example showing that universal forms may exist over a totally real field of infinite degree.
A more elaborate class of examples will be given in \Cref{prop:universal-qf-of-degree-2power}.

\begin{eg}\label{eg:pythagorean-closure}
Let $K = \qq^{\operatorname{pyth}}$ be the \emph{Pythagorean closure} of $\qq$, i.e.~the smallest subfield of $\rr$ in which every sum of squares is a square.
This is a totally real field: this follows from the fact that if $K'$ is any totally real number field and $a \in K'$ is a sum of squares, then also $K'(\sqrt{a})$ is totally real.

We claim that $X^2$ is a universal quadratic form over $K$.
Indeed, let $x \in \mc{O}_K$ be totally positive.
By Artin's Theorem \cite[Theorem VIII.I.12]{Lam}, $x$ is a sum of squares in $K$, and then in fact $x$ is a square in $K$ by the construction of $K$.
Since $\mc{O}_K$ is integrally closed, $x$ is a square also in $\mc{O}_K$, as we wanted to show.
\end{eg}

\section{Northcott property}\label{sect:northcott}

Let $K$ be an algebraic extension of $\qq$ and let $H:K\rightarrow\mathbb R_{\geq 0}$ be a height function on  $K$. 
Let us not define formally what we mean by a general height $H$; instead, we just 
recall the two examples with which we will be working. 

Let $\alpha_1,\dots,\alpha_n$ be all the conjugates of an element $\alpha\in K$. We define the \textit{house} of $\alpha$ as $\house{\alpha}= \max_i(|\alpha_i|)$.  

The \textit{logarithmic Weil height} of $\alpha\in K$ is 
$$h(\alpha)=\frac1{[\qq(\alpha):\qq]}\sum_v[\qq(\alpha)_v:\qq_v]\log\max(1,|\alpha|_v),$$
where the sum runs over all the places $v$ of $\qq(\alpha)$, $\qq(\alpha)_v$ and $\qq_v$ are the corresponding completions of $\qq(\alpha)$ and $\qq$, and $|\alpha|_v$ is the corresponding normalized absolute value.

However, in our case of interest when $\alpha$ is an algebraic integer in a totally real field, the formula simplifies to 
$$h(\alpha)=\frac1{[\qq(\alpha):\qq]}\sum_{i=1}^n\log\max(1,|\alpha_i|).$$
If $\alpha\neq 0$, we clearly have 
$h(\alpha)\leq\log\house{\alpha}$ (cf. \cite[Lemma 7]{Vidaux-Videla}).

One defines the \textit{Northcott number} of a subset $S\subseteq K$ with respect to $H$ to be 
$$\mathcal N_H(S)=\inf\left\{t\in\mathbb R_{\geq 0}\mid \#\{\alpha\in S\mid H(\alpha)<t\}=\infty\right\}$$
(as usual, $\inf\emptyset=\infty$).
Here, we follow the framework of \cite{PTW}, which generalises earlier work of \cite{Vidaux-Videla,Vidaux-Videla15} (we point out that in this earlier work, the terminology `Northcott number' is reserved for the logarithmic Weil height, and the term `Julia Robinson number' is used instead when the house is considered).
If $\mathcal N_H(S)=\infty$, then we say $S$ has the \textit{Northcott property with respect to $H$} (NP).
A weaker property, introduced in \cite{Vidaux-Videla15}, is the \textit{Julia Robinson property} (JR): this is when either $\mc{N}_H(S) = \infty$, or the infimum in the definition of $\mc N_H$ is attained as minimum.

For recent developments in the study of the Northcott property, as well as a discussion of some other properties related to Northcott numbers and the Julia Robinson property, we refer to \cite{Checcoli-Fehm, Checcoli-Widmer, J-R-Numbers, Okazaki-Sano, Widmer11} (besides the works already mentioned in the Introduction).

Observe that by the inequality $h(\alpha)\leq\log\house{\alpha},$	
if $\mc{O}_K^+$ has NP with respect to the logarithmic Weil height $h$, then $\co_K^+$  has NP with respect to the house.

Let us now prove our main results related to the Northcott property. Actually, we shall prove the following theorem, which is a slight generalization of Theorem \ref{thm:univ-Northcott}.

\begin{thm}\label{thm:3.1}
Let $K$ be a totally real infinite extension of $\qq$, and $\mc O\subseteq\mc O_K$ a subring such that $\Frac \mc O = K$. If $\mc O^+$ has the Northcott property with respect to the house, then there exists no universal quadratic $\mc O$-lattice.
\end{thm}
\begin{proof}
Assume that $(\Lambda,Q)$ is a universal quadratic $\mc O$-lattice. By \Cref{prop:defined-over}, the quadratic $\mc O$-lattice $(\Lambda,Q)$ is defined over some number field $F$, i.e.~there is an $\mc O|_F$-lattice $(\Lambda_{\mc O|_F},Q_F)$ such that $(\Lambda,Q)$ is a scalar extension of $(\Lambda_{\mc O|_F},Q_F)$ to $\mc O$. 

We fix an $F$-basis $v_1,\ldots,v_r$ for $V_F = F\Lambda_{\mc O|_F}$, and let $u_1,\ldots, u_k$ be a set of generators of $\Lambda_{\mc O|_F}$. Then we have $u_i = \gamma_{i1}v_1+\ldots+\gamma_{ir}v_r$ for $\gamma_{ij}\in F$. By scaling the $F$-basis appropriately, we may assume $\gamma_{ij} \in \mc O|_F$. It follows that every $v\in\Lambda$ may be represented by some coordinates $(\gamma_1,\ldots,\gamma_r) \in \mc{O}^r$ with respect to the basis $v_1,\ldots, v_r$.
If such a vector represents an element $Q(v) = \beta \in \mc O^+$, then $\gamma_1,\ldots,\gamma_r,\beta$ lie in some number field $F\subseteq F' \subset K$, and so we can use \cite[Lemma 4]{KY23} that says
$$\max_{1\le i \le r} \house{\gamma_i} \le C \house{\beta}^{1/2}$$ for some constant $C>0$ that is independent of $\beta$ and $F'$.

Let $\alpha \in \mc O^+$ be arbitrary. Define $B_0 = \{\alpha\}$. We construct recursively the sets $B_i \subseteq \mc O^+$ for $i\ge 1$. For $\beta \in B_i$, since $(\Lambda, Q)$ is universal, we may find $v \in \Lambda$ with coordinates $(\gamma_1,\ldots, \gamma_r) \in \mc{O}^r$ such that $Q(v) = \beta$. Let $j=1,\dots, r$. Note that $\gamma_j$ need not be totally positive, but $$\beta_j = \gamma_j + \lfloor \house{\gamma_j} \rfloor + 1$$ is totally positive. Then we have the bounds $$\house{\beta_j} < 2\house{\gamma_j}+1, \ \ \house{\gamma_j} \le C\house{\beta}^{1/2}, \ \text{ and so } \  \house{\beta_j} < 2 C \house{\beta}^{1/2} + 1 \leq (2C + 1)\house{\beta}^{1/2}.$$ 
We define $B_{i+1} = \{ \beta_1,\dots, \beta_r \mid \beta \in B_i\} \subseteq \mc O^+$. 
(Note that above, the elements $\gamma_j$ and $\beta_j$ depend on $\beta$, and so it would be more precise to denote them $\gamma_j(\beta)$ and $\beta_j(\beta)$.)

It is then straightforward to see that for all sufficiently large $i$, we have $\house{\beta} < (2C+1)^2$ for every $\beta\in B_i$.

For the sake of contradiction, assume that $\mathcal N_{\house{\cdot}}(\mc O^+) = \infty$. Then there are only finitely many elements $\beta \in \mc O^+$ with $\house{\beta} < (2C+1)^2$. Let $M \subsetneq K$ be the number field generated by $F$ as well as the elements $\beta \in \mc O^+$ with $\house{\beta} < (2C+1)^2$. Then, for arbitrary $\alpha \in \mc O^+$, we have $B_i \subseteq \mc O_M^+$ for some sufficiently large $i$. From the construction of the sets $B_i$, we deduce iteratively that $B_0 \subseteq \mc O_M^+$, that is, $\alpha \in \mc O_M^+$. Since $\alpha\in\mc O^+$ was chosen arbitrarily, this means $K \subseteq M$, a contradiction with the assumption that $K$ has infinite degree over $\qq$. So we have $\mathcal N_{\house{\cdot}}(\mc O^+)<\infty$ as claimed.
\end{proof}
We obtain a large class of examples of totally real fields without universal quadratic lattice. To state these in the highest generality, let us generalize some notations which were introduced earlier.
For $d \in \nat$ and a number field $K$, we denote by $\multiBZ[K]{d}$ the compositum of all extensions of $K$ of degree at most $d$, and by $\multiBZab[K]{d}$ the maximal abelian subextension of $\multiBZ[K]{d}/K$.
\begin{thm}\label{thm:newmain-general}
Let $d$ be a positive integer, $K$ a number field, and $L$ a totally real subfield of $\multiBZab[K]{d}$ of infinite degree over $\qq$. Then there are no universal quadratic lattices over $L$.
\end{thm}
\begin{proof}
Bombieri and Zannier \cite[Theorem 1]{Bombieri-Zannier} showed that $\multiBZab[K]{d}$ has the Northcott property with respect to the Weil height.

If we have a totally real field $L \subseteq \multiBZab[K]{d}$, then $L$ has the Northcott property with respect to the Weil height (for this property is obviously preserved under taking subfields).
By the inequality $h(\alpha)\leq\log\house{\alpha}$ that is valid for all non-zero $\alpha\in\co_L^+$ we see that $\co_L^+$ has the Northcott property with respect to the house, and so we can apply Theorem \ref{thm:univ-Northcott}.
\end{proof}
\noindent We obtain Theorem $\ref{thm:newmain}$ as a quick corollary.
\begin{proof}[Proof of Theorem $\ref{thm:newmain}$]
Let $p$ be a prime number and $q=p$ or $q = p^2$. Every finite group of order $q$ is abelian, and so every Galois extension of $\qq$ of degree $q$ is abelian, and so also the compositum $\multi{q}$ is abelian and thus a subfield of $\multiBZab{q}$. Now we can apply \Cref{thm:newmain-general}.
\end{proof}

\begin{rem}
There exist many totally real fields $K$ with the Northcott property with respect to the Weil height and which are not contained in $\multiBZ{d}$ for any $d \in \nat$: in \cite[Theorem 1.2, Corollary 1.3]{Checcoli-Fehm} it is shown, building on a criterion developed in \cite{Widmer11}, that every countable product of finite solvable groups can be realised as the Galois group of a totally real Galois extension of $\qq$ with the Northcott property with respect to the Weil height.
Also over such fields, in view of \Cref{thm:3.1}, there exist no universal quadratic lattices.
\end{rem}

\section{Totally positive units and classical quadratic lattices}
For a totally real field $K$, $\tpu_K$ denotes the group of totally positive units.
The following two propositions are well-known in the case where $K/\qq$ is a finite extension, and we give a proof to verify that everything goes through also in the infinite case.

\begin{prop}\label{prop:splitting-units} Let $K/\qq$ be an algebraic field extension.
Let $(\Lambda, Q)$ be a classical quadratic $\mc{O}_K$-lattice and $v \in \Lambda$ such that $Q(v) \in \mc{O}_K^\times$.
Then there exists a sublattice $\Lambda' \subseteq \Lambda$ such that $\Lambda = \mc{O}_K v \perp \Lambda'$.
\end{prop}

\begin{proof}
We redo the proof of \cite[Proposition 3.4]{KalaSurvey}. Set $\Lambda' = \lbrace w \in \Lambda \mid B_Q(v, w) = 0 \rbrace$. We have that $\mc{O}_K v$ is orthogonal to $\Lambda'$ and that $\mc{O}_K v \cap \Lambda' = \lbrace 0 \rbrace$. We will show that $\Lambda = \mc{O}_Kv + \Lambda'$.

For this, let $z \in \Lambda$ arbitrary. Set $w = z - B_Q(z, v)Q(v)^{-1}v$. Since $(\Lambda, Q)$ is classical, we have $B_Q(z, v)Q(v)^{-1}\in\mc O_K$, so $w\in\Lambda$. Furthermore, we compute that $B_Q(v, w) = 0$, whereby $w \in \Lambda'$. Thus $z = B_Q(z, v)Q(v)^{-1}v + w \in \mc{O}_Kv + \Lambda'$ as desired.

It remains to show that $\Lambda'$ is finitely generated. Suppose $\Lambda$ is generated by $v_1, \ldots, v_k$. Set $v_i' = v_i - B(v_i, v)Q(v)^{-1}v$ for $i \in \lbrace 1, \ldots, k \rbrace$. By the argument above, we have $v_i' \in \Lambda'$, and it follows that $\Lambda'$ is generated by $v_1', \ldots, v_k'$.
\end{proof}

We are now ready to prove a more precise version of \Cref{prop:intro}.

\begin{prop}\label{prop:units-no-classical}
Let $K$ be a totally real field and $m \in \nat$. If $\lvert \tpu_K/\mc{O}_K^{\times 2} \rvert \geq m$, then every classical universal quadratic lattice over $K$ has rank at least $m$. 
\end{prop}
\begin{proof}
Let $\varepsilon_1, \ldots, \varepsilon_m \in \tpu_K$ be representatives of distinct classes in $\tpu_K/\mc{O}_K^{\times 2}$.
Let $(\Lambda, Q)$ be a universal classical quadratic lattice.
We will show by induction on $k \in \lbrace 0, \ldots, m \rbrace$ that there exist $v_1, \ldots, v_k \in \Lambda$ with $Q(v_i) = \varepsilon_i$ and an $\mc{O}_K$-submodule $\Lambda_k'$ such that $\Lambda = \mc{O}_Kv_1 \perp \ldots \perp \mc{O}_Kv_k \perp \Lambda_k'$. This then implies in particular that $(\Lambda, Q)$ has rank at least $k$, establishing the proposition.

It remains to do the induction. For $k = 0$ there is nothing to show, and one may take $\Lambda_0' = \Lambda$.

Assume that $k > 0$; by the induction hypothesis, we may assume that $\Lambda = \mc{O}_K v_1 \perp \ldots \perp \mc{O}_Kv_{k-1} \perp \Lambda_{k-1}'$ where $Q(v_i) = \varepsilon_i$.
Since $(\Lambda, Q)$ is universal, it represents $\varepsilon_k$.
Hence there exist $x_1, \ldots, x_{k-1} \in \mc{O}_K$ and $v_k \in \Lambda_{k-1}'$ such that
$$ \varepsilon_k = Q(x_1v_1 + \ldots + x_{k-1}v_{k-1} + v_k) = \varepsilon_1x_1^2 + \ldots + \varepsilon_{k-1}x_{k-1}^2 + Q(v_k).$$
Let $K_0$ be the subfield of $K$ generated by the $\varepsilon_i, x_i$, and $Q(v_k)$.
Then $K_0$ is a number field.
Since $\varepsilon_k$ is a totally positive unit, it is \emph{indecomposable} in $K_0$, i.e.~it cannot be written as a sum of non-zero totally positive elements of $\mc{O}_{K_0}$ \cite[Lemma 2.1]{KalaYatsyna}.
Since $\varepsilon_k$ does not lie in the square class of $\varepsilon_i$ for any $i \in \lbrace 1, \ldots, k-1 \rbrace$, we conclude that we must have $\varepsilon_k = Q(v_k)$.

Since $(\Lambda, Q)$ is classical, the same holds for its sublattice $(\Lambda'_{k-1}, Q\vert_{\Lambda_{k-1}'})$.
By \Cref{prop:splitting-units} we have that $\Lambda_{k-1}' = \mc{O}_Kv_k \perp \Lambda_k'$ for some submodule $\Lambda_k' \subseteq \Lambda_{k-1}'$.
We see that indeed $\Lambda = \mc{O}_Kv_1 \perp \ldots \perp \mc{O}_Kv_k \perp \Lambda_k'$, finishing the induction.
\end{proof}

When $K/\qq$ is a totally real number field, the quantity $\lvert \tpu_K/\mc{O}_K^{\times 2} \rvert$ can be computed alternatively as follows.
Let $n = [K : \qq]$ and let $\tau_1, \ldots, \tau_n$ be the different embeddings $K \to \rr$.
Consider the map
$$ \sgn : \mc{O}_K^\times/\mc{O}_K^{\times 2} \to \lbrace \pm 1 \rbrace^n : a \mapsto (\sgn(\tau_1(a)), \ldots, \sgn(\tau_n(a)),$$
where, for a non-zero real number $r$, $\sgn(r) = 1$ if $r > 0$ and $\sgn(r) = -1$ if $r < 0$. This is a homomorphism of $\ff_2$-vector spaces, in particular, the image of the map is a subspace of $\lbrace \pm 1 \rbrace^n$. Then 
$$\lvert \tpu_K/\mc{O}_K^{\times 2} \rvert = 2^{n - \dim_{\ff_2}(\sgn(\mc{O}_K^\times/\mc{O}_K^{\times 2}))}.$$ 
As such, the quantity $\log_2 \lvert \tpu_K/\mc{O}_K^{\times 2} \rvert$ is often referred to as the \emph{unit signature rank deficiency} of $K$.
See e.g.~\cite[Section 2]{SignatureRanks1} for a discussion.

\begin{prop}\label{prop:totally-positive-finite-extension}
Let $L/K$ be a finite extension of totally real number fields.
Then $$\lvert \tpu_L / \mc{O}_L^{\times 2} \rvert \geq \lvert \tpu_K / \mc{O}_K^{\times 2} \rvert.$$
\end{prop}

\begin{proof}
See \cite[Remark 1]{SignatureRanks1}.
\end{proof}

\begin{prop}\label{prop:totally-positive-odd-extension}
Let $L/K$ be an  extension of totally real fields and assume that $\mc{O}_L^{\times 2} \cap K^\times = \mc{O}_K^{\times 2}$.
Then $\lvert \tpu_L / \mc{O}_L^{\times 2} \rvert \geq \lvert \tpu_K / \mc{O}_K^{\times 2} \rvert$.
In particular, this is the case when all finitely generated subextensions of $L/K$ have odd degree.
\end{prop}
\begin{proof}
It suffices to observe that, under the hypothesis $\mc{O}_L^{\times 2} \cap K^\times = \mc{O}_K^{\times 2}$, the restriction homomorphism $\tpu_K / \mc{O}_K^{\times 2} \to \tpu_L / \mc{O}_L^{\times 2}$ is injective.
\end{proof}

\section{Units in quadratic fields}
In this section, we collect some standard facts about quadratic extensions of number fields, and also establish some notation.

For a finite field extension $L/K$, we denote by $\Norm_{L/K}$ and $\Tr_{L/K}$ the \textit{norm} and \textit{trace} maps $L \to K$.
We first give the following general Galois theoretic observation, which is an adaptation of \cite[Remark VII.3.9]{Lam}.
\begin{lem}\label{lem:quadratic-extension-units}
Let $L/K$ be a separable quadratic field extension and $\alpha \in L$ such that $\Norm_{L/K}(\alpha) = 1$.
Then there exists $\beta \in L$ such that $\alpha\beta^2 = \Tr_{L/K}(\alpha + 1)$.

In particular, we have $\Norm_{L/K}^{-1}(K^{\times 2}) = L^{\times 2} K^{\times}$.
\end{lem}
\begin{proof}
Let $\alpha'$ denote the conjugate of $\alpha$ in $L$ and set $\beta = \alpha' + 1$.
We compute that
$$ \alpha \beta^2 = \alpha \alpha'^2 + 2\alpha \alpha' + \alpha = \alpha + \alpha' + 2 = \Tr_{L/K}(\alpha + 1)$$
as desired.

We prove the second part of the lemma. Clearly we have $\Norm_{L/K}(L^{\times 2} K^\times) \subseteq K^{\times 2}$. For the reverse inclusion, we suppose $\Norm_{L/K}(\alpha) = \gamma^2$ for some $\alpha\in L$, $\gamma\in K$. Then we have $\Norm_{L/K}(\alpha\gamma^{-1}) = 1$. By the first part of the lemma, this says $\alpha\gamma^{-1} \in L^{\times 2} K^\times$, and hence $\alpha \in L^{\times 2} K^\times$. So we have $\Norm_{L/K}^{-1}(K^{\times 2}) = L^{\times 2} K^{\times}$ as claimed.
\end{proof}

\begin{propdef}\label{propdef:quadratic-extensions}
Let $D \in \nat_{\geq 2}$ be square-free and set $K = \qq(\sqrt{D})$.
\begin{enumerate}
\item The \emph{discriminant} of $K$ is the integer
$$\Delta_K = \begin{cases}
D &\text{ if } D \equiv 1 \pmod 4 \\
4D &\text{ if } D \equiv 2, 3 \pmod 4
\end{cases}.$$
\item There exists a unique element $\varepsilon_K \in \mc{O}_K^\times$ such that $\lbrace -1, \varepsilon_K \rbrace$ generates the group $\mc{O}_K^\times$ and such that $\varepsilon_K > 1$, called the \emph{fundamental unit} of $K$. If $D\equiv 2,3 \pmod{4}$ (resp. $D\equiv 1 \pmod{4}$), this fundamental unit is given by $x + y\sqrt{D}$ where $(x, y)$ is the smallest positive solution in $\nat^2$ (resp. $(\frac 12\nat)^2$) to the equation $X^2 - DY^2 =  \pm1$.
The following are equivalent:
\begin{enumerate}[(i)]
\item\label{it:qext-1} $\varepsilon_K$ is totally positive,
\item\label{it:qext-2} $\tpu_K \neq \mc{O}_K^{\times 2}$,
\item\label{it:qext-3} $\Norm_{K/\qq}(\varepsilon_K) = 1$,
\item\label{it:qext-4} the equation $X^2 - DY^2 = -1$ has no solution in $\zz^2$.
\end{enumerate}
In particular, the above holds when $p \mid D$ for some prime $p \equiv 3 \pmod 4$.
\item Assuming $\Norm_{K/\qq}(\varepsilon_K) = 1$, let $\delta_K \in \zz$ be the unique square-free integer such that $\Tr_{K/\qq}(\varepsilon_K + 1) \delta_K \in \qq^{\times 2}$.
We have $\delta_K\varepsilon_K \in K^{\times 2}$.
Furthermore, $\delta_K \mid \Delta_K$, and $\delta_K \neq 1, \Delta_K$.
\end{enumerate}
\end{propdef}
\begin{proof}
The existence, uniqueness and computation of the fundamental unit is well-known and can be found in any book on the Pell equation, see e.g.~\cite[Sections 1.3, 4.3]{SolvingPell}.
We briefly discuss the four equivalences in (2).

$\eqref{it:qext-1} \Rightarrow \eqref{it:qext-2},\eqref{it:qext-3},\eqref{it:qext-4}$:
It follows from the defining property that $\varepsilon_K \not\in \mc{O}_K^{\times 2}$, so if $\varepsilon_K$ is totally positive, then $\varepsilon_K \in \tpu_K \setminus \mc{O}_K^{\times 2}$.
Denoting by $\sigma$ the non-trivial automorphism of $K$, then we have $\varepsilon_K, \sigma(\varepsilon_K) > 0$ and hence $0 < \varepsilon_K\sigma(\varepsilon_K) = \Norm_{K/\qq}(\varepsilon_K) \in \lbrace \pm 1 \rbrace$, i.e.~$\Norm_{K/\qq}(\varepsilon_K)=1$.
Finally, it follows that there exist no elements of $K$ with norm $-1$, which translates to the equation $X^2 - DY^2 = -1$ not having a solution in $\zz^2$.

$\neg \eqref{it:qext-1} \Rightarrow \neg\eqref{it:qext-2},\neg\eqref{it:qext-3},\neg\eqref{it:qext-4}$:
Assume that  $\varepsilon_K$ is not totally positive.
Since every element of $\mc{O}_K^{\times}$ is of the form $\pm \varepsilon_K^i$ for $i \in \zz$, we infer that $\tpu_K = \mc{O}_K^{\times 2}$.
Further, we have $0 > \varepsilon_K\sigma(\varepsilon_K) = \Norm_{K/\qq}(\varepsilon_K) \in \lbrace \pm 1 \rbrace$, which implies $\Norm_{K/\qq}(\varepsilon_K) = -1$.
If $\varepsilon_K \in \zz[\sqrt{D}]$, then $X^2 - DY^2 = -1$ has a solution in $\zz^2$.
If $\varepsilon_K \not\in \zz[\sqrt{D}]$ and $D\equiv 5\pmod 8$, then one quickly computes that $\varepsilon_K^3$ gives a solution in $\zz^2$ to $X^2 - DY^2 = -1$.
Finally, when $D\equiv 1\pmod 8$, then a mod $8$  consideration of the equation
$X^2 - DY^2 = -4$ shows that $\varepsilon_K \in \zz[\sqrt{D}]$.

For the statements in (3), the fact that $\delta_K\varepsilon_K \in K^{\times 2}$ follows from \Cref{lem:quadratic-extension-units}.
The fact that $\delta_K \mid \Delta_K$ and $\delta_K \neq 1, D$ is \cite[Proposition 1]{SignatureRanks2} (essentially \cite[Hilfssatz 8]{Kubota}).
\end{proof}
\begin{rem}
For a quadratic extension $K = \qq(\sqrt{D})$, the element $\delta_K$ can also be computed as follows: by Hilbert's Theorem 90, there exists $\rho \in \mc{O}_K$ such that $\varepsilon = \rho\sigma(\rho)^{-1}$, where $\sigma$ is the non-trivial automorphism of $K$.
Normalize $\rho$ so that it is not divisible by a rational prime.
Then $\delta_K = \Norm_{K/\qq}(\rho)$. \cite[Hilfssatz 8]{Kubota}
\end{rem}

The above result can already be used to construct examples of infinite \emph{multiquadratic fields}, i.e.~compositums of quadratic number fields, with infinite $\tpu_K / \mc{O}_K^{\times 2}$ (and thus without universal classical quadratic lattices by \Cref{prop:intro}), since we can test whether a fundamental unit $\varepsilon_K$ of a quadratic extension $K/\qq$ becomes a square in a field extension $L/K$: it is necessary and sufficient that $\delta_K$ is a square in $L$.
In particular, if $K$ is a number field and $L = K(\sqrt{D} \mid D \in S)$ for some set $S \subseteq \nat$, then an element $d \in K$ becomes a square in $L$ if and only if $d = (\prod_{D \in S_0} D) e^2$ for some $S_0 \subseteq S$ finite and $e \in K$ (by basic Kummer Theory, see e.g.~\cite[Proposition 10]{SignatureRanks2}).
\begin{eg}
Let $(n_k)_{k \in \nat}$ be a sequence of non-zero natural numbers such that the sequence $(4n_k^2 - 1)_{k \in \nat}$ consists of elements which are pairwise coprime and such that $2n_k+1$ is not a rational square for $k \in \nat$.
Consider the field $K = \qq(\sqrt{4n_k^2 - 1} \mid k \in \nat)$.
We claim that $\tpu_K/\mc{O}_K^{\times 2}$ is infinite.
In particular, by \Cref{prop:intro}, there exist no classical universal quadratic lattices over $K$.

To see this, we observe that, for $k \in \nat$, we have $\varepsilon_k = 2n_k + \sqrt{4n_k^2 - 1} \in \mc{O}_K^\times$ and, by \Cref{lem:quadratic-extension-units}, $\varepsilon_k$ lies in the same square class as $2(2n_k+1)$.
One verifies that, for $k \neq l$, neither $2(2n_k+1)$ nor $(2n_k+1)(2n_l+1)$ is a square in $K$.
Hence, $(\varepsilon_k)_{k \in \nat}$ is an infinite family of representatives of different classes in $\tpu_K / \mc{O}_K^{\times 2}$, showing that the latter set is infinite.

To find a sequence $(n_k)_{k \in \nat}$ with the desired properties, one can, for example, take $n_0 = 1$ and recursively choose $n_k$ to be any multiple of $\prod_{l < k} (4n_l^2 - 1)$ such that $2n_k + 1$ is not a square.
\end{eg}
\begin{eg}[adaptation of {\cite[Theorem 12]{SignatureRanks2}}]
Let $(q_i)_{i \in \nat}$ be a sequence of distinct primes congruent to $3 \pmod 4$.
For $i \in \nat$, let $K_i = \qq(\sqrt{q_{2i}q_{2i+1}})$.
Let $K$ be the compositum of all $K_i$.
Then $\tpu_K/\mc{O}_K^{\times 2}$ is infinite.

Indeed, by \Cref{propdef:quadratic-extensions}, for $K_i$, we compute that $\Delta_{K_i} = q_{2i}q_{2i+1}$, whence $\delta_{K_i}$ is either $q_{2i}$ or $q_{2i+1}$.
But every square-free integer which becomes a square in $K$ is either divisible by both $q_{2i}$ and $q_{2i+1}$, or is divisible by neither.
Hence, $\delta_{K_i}, \delta_{K_i}\delta_{K_j} \in \tpu_K \setminus \mc{O}_K^{\times 2}$ for all $i \neq j$, whereby $ \tpu_K/\mc{O}_K^{\times 2}$ is infinite.
\end{eg}
In both of the examples above, it is crucial that we have an infinite number of totally positive units in quadratic subextensions of our infinite multiquadratic field $K$ which do not become squares in $K$.
However, in view of \Cref{lem:quadratic-extension-units}, every totally positive unit lying in a real quadratic field becomes a square in $\multi{2}$, so such a technique is not applicable for showing that $\tpu_{\multi{2}}/\mc{O}_{\multi{2}}^{\times 2}$ is infinite.
In the next section, we look further and investigate totally positive units in biquadratic extensions of $\qq$.

\section{Units in biquadratic fields}\label{sect:biquadratic}
By a \emph{biquadratic (number) field}, we mean a field of the form $\qq(\sqrt{D_1}, \sqrt{D_2})$ for $D_1, D_2 \in \nat$ such that $D_1, D_2$, and $D_1D_2$ are not squares.
The field $\qq(\sqrt{D_1}, \sqrt{D_2})$ is then a Galois extension of $\qq$ with Galois group $(\zz/2\zz)^2$; its quadratic subfields are precisely $\qq(\sqrt{D_1})$, $\qq(\sqrt{D_2})$ and $\qq(\sqrt{D_1D_2})$.
\begin{prop}\label{prop:biquadratic-Galois}
Let $K$ be a biquadratic number field with quadratic subfields $K_1, K_2, K_3$, and $\mu \in K^\times$.
Then $\mu^2 = \Norm_{K/\qq}(\mu)^{-1} \prod_{i=1}^3 \Norm_{K/K_i}(\mu)$.
In particular, if $\mu \in \mc{O}_K^\times$, then $\mu^2 = \pm \prod_{i=1}^3 \Norm_{K/K_i}(\mu) \in \prod_{i=1}^3 \mc{O}_{K_i}^\times$.
\end{prop}
\begin{proof}
For $i \in \lbrace 1, 2, 3 \rbrace$, denote by $\sigma_i$ the automorphism of $K$ with fixed field $K_i$.
By Galois Theory we have $\Norm_{K/K_i}(\mu) = \mu \sigma_i(\mu)$ and $\Norm_{K/\qq}(\mu) = \mu \sigma_1(\mu) \sigma_2(\mu) \sigma_3(\mu)$.
So we may write
$$ \mu^2 = \frac{(\mu \sigma_1(\mu))(\mu \sigma_2(\mu))(\mu \sigma_3(\mu))}{\mu\sigma_1(\mu)\sigma_2(\mu)\sigma_3(\mu)} = \frac{\Norm_{K/K_1}(\mu)\Norm_{K/K_2}(\mu)\Norm_{K/K_3}(\mu)}{\Norm_{K/\qq}(\mu)}. $$
To deduce the second statement from this, we recall that norms of units are units, and thus $\Norm_{K/\qq}(\mu) \in \mc{O}^\times_\qq = \lbrace \pm 1 \rbrace$.
\end{proof}
\begin{rem}\label{prop:totally-positive-quadratic-subfields}
Note that, if $K/\qq$ is a totally real Galois extension, then an element $\alpha \in K$ is totally positive or totally negative if and only if $\alpha \sigma(\alpha) > 0$ for all $\sigma \in \Gal(K/\qq)$.
%Indeed, if $\alpha$ is not totally positive or totally negative, then without loss of generality there exists $\sigma \in \Gal(K/\qq)$ with $\alpha < 0$ and $\sigma(\alpha) > 0$, whereby $\alpha \sigma(\alpha) < 0$.
%The other implication is obvious.
\end{rem}
\begin{cor}\label{cor:biquadratic-new-units}
Let $K$ be a biquadratic number field with quadratic subfields $K_1, K_2, K_3$.
Assume that at least one of $K_1, K_2, K_3$ contains a totally positive non-square unit.
Let $\alpha \in \mc{O}_K^\times$, and assume that $\alpha$ is not totally negative.
The following are equivalent:
\begin{enumerate}
\item\label{it:biquad-1} $\alpha \in \tpu_K$ and $\alpha = \varepsilon_1 \varepsilon_2 \varepsilon_3$ for $\varepsilon_i \in \mc{O}_{K_i}^\times$,
\item\label{it:biquad-2} $\Norm_{K/K_i}(\alpha) \in K_i^{\times 2}$ for each $i \in \lbrace 1, 2, 3 \rbrace$,
\item\label{it:biquad-3} $\alpha \in K^{\times 2} \qq^\times$.
\end{enumerate}
Furthermore, in this case, $\varepsilon_1, \varepsilon_2, \varepsilon_3$ in \eqref{it:biquad-1} can be chosen to be all totally positive.
\end{cor}
\begin{proof}
$\eqref{it:biquad-2} \Rightarrow \eqref{it:biquad-1}$: 
In view of \Cref{prop:totally-positive-quadratic-subfields}, $\alpha$ is totally positive.
Since $\Norm_{K/\qq}(\alpha) = 1$, \Cref{prop:biquadratic-Galois} immediately yields that $\alpha = \prod_{i=1}^3 \sqrt{\Norm_{K/K_i}(\alpha)}$, whence we may set $\varepsilon_i = \sqrt{\Norm_{K/K_i}(\alpha)}$.

$\eqref{it:biquad-3} \Rightarrow \eqref{it:biquad-2}$:
If $\alpha = \beta^2 q$ for $\beta \in K^\times$ and $q \in \qq^\times$, then $\Norm_{K/K_i}(\alpha) = \Norm_{K/K_i}(\beta)^2 q^2 \in K_i^{\times 2}$ for each $i \in \lbrace 1, 2, 3 \rbrace$.

$\eqref{it:biquad-1} \Rightarrow \eqref{it:biquad-2}\& \eqref{it:biquad-3}$:
We compute that
\begin{align*}
\Norm_{K/K_1}(\alpha) &= \Norm_{K/K_1}(\varepsilon_1 \varepsilon_2 \varepsilon_3) = \Norm_{K/K_1}(\varepsilon_1) \Norm_{K/K_1}(\varepsilon_2) \Norm_{K/K_1}(\varepsilon_3) \\
&= \varepsilon_1^{ 2} \Norm_{K_2/\qq}(\varepsilon_2) \Norm_{K_3/\qq}(\varepsilon_3) \in \pm K_1^{\times 2}.
\end{align*}
Since $\alpha$ is totally positive, the same must hold for $\Norm_{K/K_1}(\alpha)$, whence $\Norm_{K/K_1}(\alpha) \in K_1^{\times 2}$.
By symmetry, $\Norm_{K/K_2}(\alpha) \in K_2^{\times 2}$ and $\Norm_{K/K_3}(\alpha) \in K_3^{\times 2}$.
This shows that \eqref{it:biquad-2} holds. We also infer that $\Norm_{K_1/\qq}(\varepsilon_1) = \Norm_{K_2/\qq}(\varepsilon_2) = \Norm_{K_3/\qq}(\varepsilon_3)$.
By the assumption that at least one of $K_1, K_2, K_3$ contains a totally positive non-square unit, we cannot have that $\Norm_{K_1/\qq}(\varepsilon_1) = \Norm_{K_2/\qq}(\varepsilon_2) = \Norm_{K_3/\qq}(\varepsilon_3) = -1$ (see \Cref{propdef:quadratic-extensions}), so we must have $\Norm_{K_1/\qq}(\varepsilon_1) = \Norm_{K_2/\qq}(\varepsilon_2) = \Norm_{K_3/\qq}(\varepsilon_3)=1$.
In view of \Cref{prop:totally-positive-quadratic-subfields} we may thus assume without loss of generality that all $\varepsilon_i$ are totally positive.
Furthermore, $\varepsilon_i \in K_i^{\times 2} \qq^\times$ by \Cref{lem:quadratic-extension-units}.
We thus conclude that \eqref{it:biquad-3} also holds.
\end{proof}
We saw in \Cref{lem:quadratic-extension-units} that for a quadratic number field $K$, one has $\mc{O}_K^{+, \times} \subseteq K^{\times 2} \qq^\times$.
In the rest of this section, we consider examples of biquadratic number fields $K$ where $\mc{O}_K^{+, \times} \subseteq K^{\times 2} \qq^\times$ does not hold, and an example where it does hold.
\begin{cor}\label{cor:biquadratic-new-units-eg}
Let $K$ be a biquadratic number field with quadratic subfields $K_1, K_2, K_3$.
Suppose that
\begin{enumerate}
\item there exists $i \in \lbrace 1, 2, 3 \rbrace$ such that $\tpu_{K_i} \neq \mc{O}_{K_i}^{\times 2}$,
\item for each $i \in \lbrace 1, 2, 3 \rbrace$ we have $\tpu_{K_i} \subseteq K^{\times 2}$.
\end{enumerate}
Then $\tpu_K \not\subseteq K^{\times 2} \qq^\times$.
\end{cor}
\begin{proof}
By the first condition, we have (without loss of generality) that $\tpu_{K_1} \neq \mc{O}_{K_1}^{\times 2}$.
By \Cref{prop:totally-positive-finite-extension} it follows that there exists $\alpha \in \tpu_K \setminus \mc{O}_K^{\times 2}$.
Suppose, for the sake of a contradiction, that $\alpha \in K^{\times 2}\qq^\times$.
By \Cref{cor:biquadratic-new-units}, $\alpha = \varepsilon_1\varepsilon_2\varepsilon_3$ for $\varepsilon_i \in \tpu_{K_i}$.
But we assumed that $\tpu_{K_i} \subseteq K^{\times 2}$, contradicting the assumption that $\alpha \not\in \mc{O}_K^{\times 2}$. Hence $\alpha \in \tpu_K \setminus K^{\times 2} \qq^\times$.
\end{proof}
\begin{eg}
Let $p$ be a prime with $p \equiv 3 \pmod 4$ and consider $K = \qq(\sqrt{2}, \sqrt{p})$.
In $K_1=\qq(\sqrt{2})$ there are no totally positive non-square units.
In $K_2=\qq(\sqrt{p})$ and $K_3=\qq(\sqrt{2p})$ there are respective totally positive fundamental units $\varepsilon_2$ and $\varepsilon_3$ with corresponding $\delta_2, \delta_3 \in\{2,2p\} \subseteq K^{\times 2}$; in particular it follows that $\varepsilon_2, \varepsilon_3 \in K^{\times 2}$.
Hence \Cref{cor:biquadratic-new-units-eg} applies, and we conclude that $\tpu_K \not\subseteq K^{\times 2}\qq^\times$.
\end{eg}
\begin{prop}\label{prop:biquadratic-family}
Let $n \in \nat$ be such that $n \equiv 1 \pmod{12}$, and such that each of
$$ d_1 = n(n+1), \quad d_2 = 3n(3n+4), \quad d_3 = (3n+3)(3n+4) $$
are square-free integers.
Let $K = \qq(\sqrt{d_1}, \sqrt{d_2})$.
Then
$$ \mu = \frac{3n+4}{2} + \frac{3}{2}\sqrt{d_1} + \frac{1}{2}\sqrt{d_2} + \frac{1}{2}\sqrt{d_3} \in \tpu_K \setminus K^{\times 2} \qq^\times.$$
\end{prop}
\begin{proof}
For $i \in \lbrace 1, 2, 3 \rbrace$, let $K_i = \qq(\sqrt{d_i})$.
Then $K_1, K_2, K_3$ are the quadratic subfields of $K$.
We compute that
$$ \Norm_{K/K_1}(\mu) = 1, \quad \Norm_{K/K_2}(\mu) = \frac{3n+2 + \sqrt{d_2}}{2}, \quad \Norm_{K/K_3}(\mu) = 6n+7 + 2\sqrt{d_3}.$$
These are all totally positive units, hence by \Cref{prop:totally-positive-quadratic-subfields} $\mu$ is also a totally positive unit.
On the other hand, one computes that $\frac{3n+2 + \sqrt{d_2}}{2}$ is the fundamental unit of $K_2$ and hence a non-square, whence we obtain by \Cref{cor:biquadratic-new-units} that $\mu \not\in K^{\times 2} \qq^\times$.
\end{proof}
We conclude the section with an example of a biquadratic field $K$ with $\mc{O}_K^{\times 2} \subsetneq \tpu_K \subseteq K^{\times 2}\qq^\times$. 
\begin{eg}
Let $K = \qq(\sqrt{3}, \sqrt{7})$.
We show that, even though all quadratic subfields have totally positive non-square units, we have $\tpu_K \subseteq K^{\times 2}\qq^\times$. This shows that the second condition in \Cref{cor:biquadratic-new-units-eg} is necessary. 

We compute the following:
{
\begin{center}
\begin{tabular}{r||c|c|c}
& $K_1 = \qq(\sqrt{3})$ & $K_2 = \qq(\sqrt{7})$ & $K_3 = \qq(\sqrt{21})$ \\ \hline
fundamental unit & $\varepsilon_1 = 2 + \sqrt{3}$ & $\varepsilon_2 = 8+3\sqrt{7}$ & $\varepsilon_3 = \frac{5+\sqrt{21}}2$ \\
$\delta$ & $\delta_1 = 6$ & $\delta_2 = 2$ & $\delta_3 = 7$
\end{tabular}
\end{center}
}
We infer that $\varepsilon_1\varepsilon_2$ and $\varepsilon_3$ are squares in $K$, whereas $\varepsilon_1$ and $\varepsilon_2$ are not.
By \Cref{prop:biquadratic-Galois} and the fact that $\lbrace -1, \varepsilon_i \rbrace$ generates $\mc{O}_{K_i}^\times$ for $i \in \lbrace 1, 2, 3 \rbrace$, it follows that every element of $\mc{O}_K^\times$ is of the form $\pm\sqrt{\varepsilon_1\varepsilon_2}^k \sqrt{\varepsilon_3}^l \varepsilon_1^m$ for some $k, l, m \in \nat \cup \{ 0 \}$.
Since we know by \Cref{lem:quadratic-extension-units} that $\tpu_{K_i} \subseteq K^{\times 2}\qq^\times$, to show that $\tpu_K \subseteq K^{\times 2}\qq^\times$, it remains to show that $\pm\sqrt{\varepsilon_1\varepsilon_2}, \pm\sqrt{\varepsilon_3}, \pm\sqrt{\varepsilon_1\varepsilon_2\varepsilon_3} \not\in \tpu_K$.

To this end, one computes that
\begin{align*}
\sqrt{\varepsilon_1\varepsilon_2} &= \frac{3+3\sqrt{3} + \sqrt{7}+\sqrt{21}}{2}, \quad \sqrt{\varepsilon_3} = \frac{\sqrt{3}+\sqrt{7}}{2},\\
\sqrt{\varepsilon_1\varepsilon_2\varepsilon_3} &= \frac{8+5\sqrt{3}+3\sqrt{7}+2\sqrt{21}}{2},
\end{align*}
and observes that none of them are totally positive or totally negative.
\end{eg}

\section{Infinite multiquadratic fields}\label{sec:multi}
For a field $K$ and two extensions $K_1, K_2$ of $K$ contained within a fixed common overfield $\tilde{K}$, we denote by $K_1K_2$ the \emph{compositum of $K_1$ and $K_2$}, i.e.~the smallest subfield of $\tilde{K}$ which contains both $K_1$ and $K_2$.
For a family $(K_i)_{i \in I}$ of field extensions of $K$ contained in $\tilde{K}$, we similarly define its compositum to be the smallest subfield of $\tilde{K}$ containing $K_i$ for all $i \in I$.

Recall that a family $(K_i)_{i \in \nat}$ of Galois extensions of $\qq$ (seen as subfields of $\cc$) is \emph{linearly disjoint} if for every $n \in \nat$, we have $K_1K_2\cdots K_n \cap K_{n+1} = \qq$.
\begin{prop}\label{prop:quadratic-lin-disjoint}
Let $K$ be a field.
Let $K_1/K$ be a Galois extension and $K_2/K$ an algebraic extension such that, within a fixed algebraic closure of $K$, we have $K_1 \cap K_2 = K$.
Denote by $K_1K_2$ the compositum of $K_1$ and $K_2$ over $K$.
We have
$$ K_2 \cap (K_1K_2)^{\times 2}K^\times = K_2^{\times 2} K^\times. $$
\end{prop}
\begin{proof}
See \cite[Theorem VII.7.2]{Lam}.
\end{proof}

\begin{prop}\label{prop:composita-units-multiquadratic}
Let $(K_i)_{i \in \nat}$ be a linearly disjoint family of Galois extensions of $\qq$, and let $K$ be the compositum of $(K_i)_{i \in \nat}$.
Assume that there is an infinite subset $I \subseteq \nat$ such that, for all $i \in I$, $K_i$ is a biquadratic number field with $\tpu_{K_i} \not\subseteq K_i^{\times 2} \qq^\times$.
Then $\tpu_K/\mc{O}_K^{\times 2}$ is infinite.
\end{prop}
\begin{proof}
For $i \in I$, fix $\alpha_i \in \tpu_{K_i} \setminus K_i^{\times 2}\qq^\times$.
By \Cref{cor:biquadratic-new-units} there exists a quadratic subfield $K_i'$ of $K_i$ with $\Norm_{K_i/K_i'}(\alpha_i) \not\in K_i'^{\times 2}$; we also fix these subfields from now on.

The proof strategy is as follows: we will recursively (for $k = 0, 1, \ldots$) choose $i_k \in I \setminus \lbrace i_0, \ldots, i_{k-1} \rbrace$ such that, for $l \in \lbrace 0, 1, \ldots, k-1 \rbrace$ we have $\alpha_{i_l}\alpha_{i_k} \not\in (K_{i_l}K_{i_k})^{\times 2} \qq^\times$.
In view of \Cref{prop:quadratic-lin-disjoint}, we then have $\alpha_{i_l}\alpha_{i_k} \not\in K^{\times 2} \qq^\times$, and since $I$ is infinite, this is enough to conclude that $\tpu_K/\mc{O}_K^{\times 2}$ is infinite.

Let $k \in \nat$ be such that $i_l$ has been appropriately chosen for $l < k$.
By the assumption that the $(K_i)_{i \in \nat}$ are linearly disjoint, there are at most $k$ elements $i \in I \setminus \lbrace i_0, \ldots, i_{k-1} \rbrace$ for which there exists $l \in \lbrace 0, \ldots, k-1 \rbrace$ with $\Norm_{K_{i_l}/K_{i_l}'}(\alpha_{i_l}) \in (K_{i_l}'K_i)^{\times 2}$.

Hence, we may choose $i_k \in I \setminus \lbrace i_0, \ldots, i_{k-1} \rbrace$ such that $\Norm_{K_{i_l}/K_{i_l}'}(\alpha_{i_l}) \not\in (K_{i_l}'K_{i_k})^{\times 2}$ for all $l < k$.
To conclude the proof, we show that for $l \in \lbrace 0, \ldots, k-1 \rbrace$, we have $\alpha_{i_l}\alpha_{i_k} \not\in (K_{i_l}K_{i_k})^{\times 2} \qq^\times$.
But this follows by \Cref{lem:quadratic-extension-units}, since
\begin{equation*}
\Norm_{K_{i_l}K_{i_k}/K_{i_l}'K_{i_k}}(\alpha_{i_l}\alpha_{i_k}) = \Norm_{K_{i_l}/K_{i_l}'}(\alpha_{i_l})\alpha_{i_k}^2 \not\in (K_{i_l}'K_{i_k})^{\times 2}. \qedhere
\end{equation*}
\end{proof}
By combining \Cref{prop:composita-units-multiquadratic} with the examples of biquadratic fields from \Cref{sect:biquadratic}, we obtain examples of infinite multiquadratic fields $K$ for which $\tpu_K / \mc{O}_K^{\times 2}$ is infinite.
In particular, we obtain the following announced result.
\begin{thm}\label{thm:Q2}
Let $K = \multi{2}$.
Then $\tpu_K / \mc{O}_K^{\times 2}$ is infinite.
In particular, there are no universal classical quadratic lattices over $K$.
\end{thm}
\begin{proof}
Let $n_1, n_2, \ldots$ be an infinite sequence of distinct natural numbers such that for each $i \in \nat$, we have $n_i \equiv 1 \pmod{12}$, and $n_i(n_i + 1)$, $3n_i(3n_i+4)$, and $(3n_i+3)(3n_i+4)$ are square-free.
For $i \geq 1$, let $K_i = \qq(\sqrt{n_i(n_i+1)}, \sqrt{3n_i(3n_i+4)})$. Let $K_0$ be a maximal subfield of $K$ such that $K_0 \cap K_1\cdots K_m = \qq$ for any $m \in \nat$.
We have that $K$ is the compositum of the linearly disjoint family $(K_i)_{i \in \nat \cup \{ 0 \} }$.

Furthermore, for $i \geq 1$, we have by \Cref{prop:biquadratic-family} that $K_i$ is a biquadratic field with $\tpu_{K_i} \not\subseteq K_i^{\times 2}\qq^{\times}$.
We conclude by \Cref{prop:composita-units-multiquadratic} that $\tpu_K / \mc{O}_K^{\times 2}$ is infinite.

The second statement follows from the first by \Cref{prop:units-no-classical}.
\end{proof}

In \Cref{eg:pythagorean-closure} we gave an example of a totally real field with a universal quadratic form of rank $1$.
We can refine the construction to obtain more examples of totally real fields of infinite degree over $\qq$ with universal quadratic forms.
For a totally real field $K$, define
\begin{align*}
	R(K) &= \inf \lbrace n \in \nat \mid \text{ there exists a universal quadratic lattice of rank }n \text{ over } K \rbrace, \\
	\Rc(K) &= \inf \lbrace n \in \nat \mid \text{ there exists a classical universal quadratic lattice of rank }n \text{ over } K \rbrace.
\end{align*}
We clearly have $1 \leq R(K) \leq \Rc(K)$.

\begin{thm}\label{prop:universal-qf-of-degree-2power}
	Let $n \in \nat$.
	There exists a totally real field $K$ of infinite degree over $\qq$ with $\lvert K^{+, \times}/K^{\times 2} \rvert = \lvert \tpu_K / \mc{O}_K^{\times 2} \rvert = 2^n$.
	In particular, $\Rc(K) = 2^n$.
\end{thm}
\begin{proof}
	For a field $L$ and $e_1, \ldots, e_n \in L^\times$, define $\Pi_L(e_1, \ldots, e_n)$ to be the subgroup of $L^\times/L^{\times 2}$ generated by $e_1L^{\times 2}, \ldots, e_nL^{\times 2}$.
	
	Let $K_0$ be any totally real field with $\lvert \tpu_{K_0} / \mc{O}_{K_0}^{\times 2} \rvert \geq 2^n$ (in view of \Cref{thm:Q2}, one can take for example $K_0 = \multi{2}$).
	Let $e_1, \ldots, e_n \in \tpu_{K_0}$ be such that their representatives in $\tpu_{K_0} / \mc{O}_{K_0}^{\times 2}$ are linearly independent, i.e.~$\lvert \Pi_{K_0}(e_1, \ldots, e_n) \rvert = 2^n$.
	Recursively, for $i \in \nat$, let $K_{i+1}$ be the smallest subfield of $\cc$ containing $K_i$ and the square roots of all totally positive elements $\alpha \in K_i \setminus \Pi_{K_i}(e_1, \ldots, e_n) K_i^{\times 2}$.
	Define $K = \bigcup_{i \in \nat} K_i$; we claim that this field is as desired.
	Firstly, $K$ is totally real, since the square root of a totally positive element is always totally real.
	Secondly, it is clear that $K_1$, and hence also $K$, is of infinite degree over $\qq$.
	
	Observe that, by construction, $$\Pi_{K}(e_1, \ldots, e_n) K^{\times 2} = K^{+, \times}\text{ and }\lvert \Pi_{K}(e_1, \ldots, e_n) \rvert = \lvert \tpu_K/\mc{O}_K^{\times 2} \rvert = 2^n.$$
	It follows by \Cref{prop:units-no-classical} that there exist no classical universal quadratic lattices over $K$ in less than $2^n$ variables.
	
	On the other hand, we have that the diagonal quadratic form in $2^n$ variables given by
	$$\mc{O}_K^{2^n}  \to \co_K^+ : (a_I)_{I \subseteq \lbrace 1, \ldots, n \rbrace} \mapsto \sum_{I \subseteq \lbrace 1, \ldots, n \rbrace} a_I^2 \left(\prod_{j \in I} e_j\right).$$
	is universal.
\end{proof}

\subsection{Questions}\label{sec:ques}
A finite extension of $\qq$ can never have a universal quadratic lattice of rank less than $3$ (for local reasons); the passage to infinite extension of $\qq$ thus `unlocks' the values $1$ and $2$ for $R(K)$ and $\Rc(K)$ by \Cref{prop:universal-qf-of-degree-2power}.
It is a wide open question which values the invariants $R$ and $\Rc$ can take for number fields. Using the approach of Kim--Kim--Park \cite{KP2} (and its refinement by Kala--Kr\' asensk\' y--Park--Yatsyna--\. Zmija, in preparation), one can establish that the values $3,4,5,6,7,8$ are attained as $R$ and $\Rc$ of real quadratic fields. However, nothing more is probably known about precise minimal ranks of universal lattices over number fields.

Already the preceding \Cref{prop:universal-qf-of-degree-2power} shows that we can obtain more information for infinite extensions; perhaps also the more general variation of this question will be easier to approach.
\begin{ques}
	For which integers $n$ does there exist a totally real field with $\Rc(K) = n$ (respectively $R(K) = n$)?
\end{ques}

Further, it is very interesting to consider generalization of \Cref{thm:Q2} to subfields of $\qq^{(2)}$.

\begin{ques}
	Does there exist a subfield $K$ of $\qq^{(2)}$ of infinite degree over $\qq$ for which $\tpu_K/\mc{O}_K^{\times 2}$ is finite?
\end{ques}
If a very strong version of Cohen--Lenstra heuristics holds, namely (cf. \cite[Remark 14]{SignatureRanks3})
\begin{itemize}
	\item there is a positive density of real multiquadratic fields $K$ for which $\tpu_K = \mc O_K^{\times 2}$, and
	\item the subfields $K_i$ of $K$ satisfies $\tpu_{K_i} = \mc O_{K_i}^{\times 2}$ with independent probabilities,
\end{itemize}
then the heuristics predict that infinite degree fields $K\subseteq\qq^{(2)}$ for which $\tpu_K = \mc O_K^{\times 2}$ exist (although they are exceedingly rare).

\section{Composita of odd degree Galois extensions}
The following can be seen as a generalization of \cite[Proposition 14]{SignatureRanks1}, and should be compared with \Cref{prop:composita-units-multiquadratic}.
\begin{prop}\label{prop:composita-units-multi-odd}
Let $(K_i)_{i \in \nat}$ be a linearly disjoint family of Galois extensions of $\qq$, and let $K$ be the compositum of $(K_i)_{i \in \nat}$.
Assume that there is an infinite subset $I \subseteq \nat$ such that, for all $i \in I$, $[K_i : \qq]$ is odd, and $\tpu_{K_i} \not\subseteq \mc{O}_{K_i}^{\times 2}$.
Then $\tpu_K/\mc{O}_K^{\times 2}$ is infinite.
\end{prop}
\begin{proof}
For $i \in I$, fix $\alpha_i \in \tpu_{K_i}\setminus \mc{O}_{K_i}^{\times 2}$.
We will show that, for $i, j \in I$ with $i \neq j$, $\alpha_i\alpha_j \not\in (K_iK_j)^{\times 2}\qq^\times$.
In view of \Cref{prop:quadratic-lin-disjoint}, we then conclude that $\tpu_K/\mc{O}_K^{\times 2}$ is infinite.

So, pick $i, j \in I$ with $i \neq j$, and assume for the sake of arriving at a contradiction that $\alpha_{i}\alpha_{j} = \beta^2 q$ for some $q \in \qq^\times$ and $\beta \in K_{i}K_{j}$.
For an arbitrary $\sigma \in \Gal(K_{i}K_{j}/K_{j})$ we compute
$$ \sigma(\alpha_{i})\alpha_{j} = \sigma(\alpha_{i}\alpha_{j}) = q\sigma(\beta)^2.$$
So we have $\alpha_{i}\sigma(\alpha_{i})\alpha_j^2 = q^2\beta^2\sigma(\beta)^2 \in (K_{i}K_{j})^{\times 2}$, and thus $\alpha_{i}\sigma(\alpha_{i})\in (K_{i}K_{j})^{\times 2}$.
Recall that, since $K_{i}/\qq$ is Galois for each $i \in I$ and $K_{i} \cap K_j = \qq$, we have that the natural map $\Gal(K_{i}/\qq) \to \Gal(K_{i}K_{j}/K_{j})$ is a bijection.
Since $\lvert \Gal(K_i/\qq) \rvert = [K_i : \qq]$ is odd, we infer that
$$ 1 = \Norm_{K_i/\qq}(\alpha_i) = \prod_{\sigma \in \Gal(K_i/\qq)} \sigma(\alpha_i) \in \alpha_i (K_iK_j)^{\times 2} $$
whereby $\alpha_i \in (K_iK_j)^{\times 2}$.
But since $[K_iK_j : K_i] = [K_j : \qq]$ is odd, this is only possible if $\alpha_i \in K_i^{\times 2}$, which contradicts our assumption that $\alpha_i \not\in \mc{O}_{K_i}^{\times 2}$.
\end{proof}
The heuristics of \cite{CyclicCompositum} suggest that Galois number fields of odd degree are abundant.
In particular we obtain the following (largely contained in \cite[Proposition 14]{SignatureRanks1}).

\begin{thm}\label{thm:Q3}
Let $K = \qq^{(3d)}$ for $d \in \nat$ odd. Then $\tpu_K/\mc{O}_K^{\times 2}$ is infinite. In particular, there are no universal classical quadratic lattices over $K$.
\end{thm}

\begin{proof}
By \cite[Theorem 1.1.2]{CyclicCompositum} there exist infinitely many cyclic cubic extensions with a totally positive non-square unit.
Let $K_0$ be the compositum of these infinitely many cyclic cubic extensions.
In view of \Cref{prop:composita-units-multi-odd} we have that $\tpu_{K_0}/\mc{O}_{K_0}^{\times 2}$ is infinite.

We have that $K/K_0$ is a limit of odd degree extensions, so the fact that $\tpu_K/\mc{O}_K^{\times 2}$ is infinite follows from the fact that $\tpu_{K_0}/\mc{O}_{K_0}^{\times 2}$ is infinite by \Cref{prop:totally-positive-odd-extension}.

The second statement follows from the first one by \Cref{prop:units-no-classical}.
\end{proof}

%\bibliographystyle{alpha}
%\bibliography{references}
\printbibliography

\end{document}